                    \def\version{June 18, 2018}                       %

 \documentclass[reqno,11pt]{amsart}
 \usepackage{amsmath, amsthm, a4, latexsym, amssymb}
\usepackage[unicode]{hyperref}
\usepackage{color}
\usepackage{srcltx}

\def\@rmrk#1#2{\refstepcounter
    {#1}\@ifnextchar[{\@yrmrk{#1}{#2}}{\@xrmrk{#1}{#2}}}

\makeatletter\@addtoreset{equation}{section}\makeatother

 \newfont{\bfit}{cmbxti10 scaled 1200}

\renewcommand{\d}{{\rm d}}

 \newcommand{\e}{{\rm e} }

 \newcommand{\eps}{\varepsilon}

 \newcommand{\R}{\mathbb{R}}
 \newcommand{\N}{\mathbb{N}}

 \newcommand{\E}{\mathbb{E}}
 \renewcommand{\P}{\mathbb{P}}
 \def\1{{\mathchoice {1\mskip-4mu\mathrm l} 
{1\mskip-4mu\mathrm l}
{1\mskip-4.5mu\mathrm l} {1\mskip-5mu\mathrm l}}}

 \newcommand{\Mcal}{{\mathcal M}}

\newcommand{\heap}[2]{\genfrac{}{}{0pt}{}{#1}{#2}}

\newcommand{\ssup}[1] {{\scriptscriptstyle{({#1}})}}

\renewcommand{\subsection}{\secdef \subsct\sbsect}
\newcommand{\subsct}[2][default]{\refstepcounter{subsection}
\vspace{0.15cm}
{\flushleft\bf \arabic{section}.\arabic{subsection}~\bf #1  }
\nopagebreak\nopagebreak}
\newcommand{\sbsect}[1]{\vspace{0.1cm}\noindent
{\bf #1}\vspace{0.1cm}}

{\nopagebreak {\hfill\rule{2mm}{2mm}}\\ }

\newtheorem{theorem}{Theorem}[section]
\newtheorem{lemma}[theorem]{Lemma}
\newtheorem{cor}[theorem]{Corollary}

\newtheoremstyle{thm}{1.5ex}{1.5ex}{\itshape\rmfamily}{}
{\bfseries\rmfamily}{}{2ex}{}

\newtheoremstyle{rem}{1.3ex}{1.3ex}{\rmfamily}{}
{\itshape\rmfamily}{}{1.5ex}{}
\theoremstyle{rem}
\newtheorem{remark}{{\slshape\sffamily Remark}}[]

\refstepcounter{subsubsection}

\def\thebibliography#1{\section*{References}
  \list%
  {\arabic{enumi}.}
    {\settowidth\labelwidth{[#1]}\leftmargin\labelwidth
    \advance\leftmargin\labelsep
    \parsep0pt\itemsep0pt
    \usecounter{enumi}}
    \def\newblock{\hskip .11em plus .33em minus .07em}
    \sloppy                   
    \sfcode`\.=1000\relax}

\begin{document}
\title[Strong coupling limit of the Polaron measure and the Pekar process]
{\large Strong coupling limit of the Polaron measure and the Pekar process}
\author[Chiranjib Mukherjee and S. R. S. Varadhan]{}
\maketitle
\thispagestyle{empty}
\vspace{-0.5cm}

\centerline{\sc  Chiranjib Mukherjee\footnote{University of M\"unster, Einsteinstrasse 62, M\"unster 48149, Germany, {\tt chiranjib.mukherjee@uni-muenster.de}}
and S. R. S. Varadhan \footnote{Courant Institute of Mathematical Sciences, 251 Mercer Street, New York, NY 10012, USA, {\tt varadhan@cims.nyu.edu}}}
\renewcommand{\thefootnote}{}
\footnote{\textit{AMS Subject
Classification:} 60J65, 60J55, 60F10.}
\footnote{\textit{Keywords:} Polaron problem, strong coupling limit, large deviations, Pekar process}

\vspace{-0.5cm}
\centerline{\textit{University of M\"unster and Courant Institute New York}}
\vspace{0.2cm}

\begin{center}
\version
\end{center}

\begin{quote}{\small {\bf Abstract: }
 The  {\it{Polaron measure}} is defined as the transformed path measure
  $$\widehat\P_{\eps,T}= Z_{\eps,T}^{-1}\,\, \exp\bigg\{\frac{1}{2}\int_{-T}^T\int_{-T}^T\frac{\eps\e^{-\eps|t-s|}}{|\omega(t)-\omega(s)|} \,\d s \,\d t\bigg\}\d\P$$
with respect to the law $\P$ of three dimensional Brownian increments on a   finite interval $[-T,T]$, 
and $ Z_{\eps,T}$ is  the partition function with $\eps>0$ being a constant. The logarithmic asymptotic behavior of the partition function $Z_{\eps,T}$ was analyzed in  \cite{DV83} 
showing that
$$
g_0=\lim_{\eps \to 0}\bigg[\lim_{T\to\infty}\frac{\log Z_{\eps,T}}{2T}\bigg]=\sup_{\heap{\psi\in H^1(\R^3)}{\|\psi\|_2=1}} 
\bigg\{\int_{\R^3}\int_{\R^3}\d x\d y\,\frac {\psi^2(x) \psi^2(y)}{|x-y|} -\frac 12\big\|\nabla \psi\big\|_2^2\bigg\}.
$$
In \cite{MV18} we analyzed the actual path measures and showed that
 the limit ${\widehat \P}_{\eps}=\lim_{T\to\infty}\widehat\P_{\eps,T}$  exists and identified this limit  explicitly, and as a corollary, 
 we also deduced the central limit theorem for 
 $(2T)^{-1/2}(\omega(T)-\omega(-T))$
 under $\widehat\P_{\eps,T}$ and obtained an expression for the limiting variance $\sigma^2(\eps)$. 
 In the present article, we investigate the {\it{strong coupling limit}} $\lim_{\eps\to 0} \lim_{T\to\infty} \widehat\P_{\eps,T}=\lim_{\eps\to 0} \widehat \P_\eps$ and show that this limit coincides with the 
increments of the stationary  Pekar process with generator 
 $$
 \frac 12 \Delta+ \bigg(\frac{\nabla\psi}\psi\bigg)\cdot \nabla
 $$ 
 for any maximizer $\psi$  of the free enrgy $g_0$. The Pekar process was also earlier identified in \cite{MV14}, \cite{KM15} and \cite{BKM15} as the limiting object of the {\it{mean-field Polaron}} measures.}
\end{quote}


\section{Introduction and motivation}\label{sec-intro}

Questions on path measures pertaining to self-attractive random motions, or Gibbs measures on interacting random paths, are often motivated by the important r\^{o}le 
they play in quantum statistical mechanics. A problem similar in spirit to these considerations is connected with the {\it polaron problem}. 
The physical question arises from the description of the slow movement of a charged particle, e.g. an electron, in a crystal whose lattice sites are polarized by this motion, 
influencing the behavior of the electron and determining its {\it{effective behavior}}. For the physical relevance of this model, we refer to
the lectures by Feynman \cite{F72}. The mathematical layout of this problem was also founded by Feynman. Indeed, he introduced a path integral formulation of this problem and pointed out that the aforementioned effective behavior can be studied via studying a certain path measure. This measure is written in terms of a three dimensional
Brownian motion acting under a self-attractive Coulomb interaction: 
\begin{equation}\label{eq-Polaron-alpha}
\widehat\P_{\alpha,T}(\d \omega)= \frac 1 {Z_{\alpha,T}} \exp\bigg\{\frac \alpha 2\int_{-T}^T\int_{-T}^T \d\sigma \d s \frac{\e^{-|t-s|}}{|\omega(t)- \omega(s)|}\bigg\} \,\, \P(\d \omega),
\end{equation}
where $\alpha>0$ is a constant, or the {\it{coupling parameter}}, $\P$ refers to the law of three dimensional white noise, or Brownian increments, and $Z_{\alpha,T}$ is the normalization constant or partition function.
The physically relevant regime is the {\it{strong coupling limit}} which corresponds to studying the asymptotic behavior of the interaction as $T\to\infty$, followed by $\alpha\to\infty$. 
Since by Brownian scaling, for any $\alpha>0$,
\begin{equation}\label{eq-alpha-eps}
\int_{-\alpha^2 T}^{\alpha^2 T}\int_{-\alpha^2 T}^{\alpha^2 T} \frac{ \alpha^{-2}\e^{-\alpha^{-2}\,|t-s|}\,\d s \,\d t}{|\omega(t)-\omega(s)|} = \alpha \int_{-T}^T\int_{-T}^T \frac{ \e^{-|t-s|}\,\d s \,\d t}{|\omega(t)-\omega(s)|},
\end{equation}
in distribution, with setting $\eps=1/\alpha^2$ the strong coupling limit of the Polaron measure then reduces to studying its asymptotic behavior as $T\to \infty$ followed by $\eps\to0$.

The analysis of the logarithmic behavior of the partition function $Z_{\eps,T}$  was rigorously carried out  by Donsker and Varadhan (\cite{DV83}) resulting in the following {\it{free energy variational formula}} which also resolved Pekar's conjecture (\cite{P49}):
\begin{equation}\label{Pekarfor}
 \begin{aligned}
g_0&=\lim_{\eps\to 0}\lim_{T\to\infty}\frac 1{2T} \log Z_{\eps,T}\\
&=\sup_{\heap{\psi\in H^1(\R^3)}{\|\psi\|_2=1}} 
\Bigg\{\int_{\R^3}\int_{\R^3}\d x\d y\,\frac {\psi^2(x) \psi^2(y)}{|x-y|} -\frac 12\big\|\nabla \psi\big\|_2^2\Bigg\},
\end{aligned}
\end{equation}
with $H^1(\R^3)$ denoting the usual Sobolev space of square integrable
functions with square integrable gradient. The variational formula in \eqref{Pekarfor} was analyzed by Lieb (\cite{L76}) with the result that there is a maximizer $\psi_0$, which is unique modulo spatial shifts and is rotationally symmetric. 

Questions pertaining to the asymptotic behavior of the actual path measures $\widehat\P_{\eps,T}$, however, had remained 
unanswered on a rigorous level. In a recent article (\cite{MV18}) we have shown that there exists $\eps_0,\eps_1\in (0,\infty)$ such that if $\eps\in (0,\eps_0)\cup (\eps_1,\infty)$, 
the limit ${\widehat \P}_{\eps}=\lim_{T\to\infty}\widehat\P_{\eps,T}$  exists and identified the limit $\widehat\P_\eps$ {\it{explicitly}}. 
As a corollary, we have also deduced the central limit theorem for 
 $$
 \frac{1}{\sqrt{2T}}(\omega(T)-\omega(-T))
 $$
for the increment process under $\widehat\P_{\eps,T}$ and obtained an expression for the limiting variance $\sigma^2(\eps)\in (0,1]$. 
It is the goal of the present article to investigate and explicitly characterize the strong coupling limit 
$$
\lim_{\eps\to 0} \widehat \P_\eps.
$$

Before we turn to a more formal description of our main results, it is useful to provide an intuitive interpretation of \eqref{Pekarfor}: We remark that the interaction appearing in the Polaron problem {\it{self-attractive}} - 
for fixed $\eps>0$, the Polaron measure $\widehat\P_{\eps,T}$ favors 
paths which make $|\omega(t)- \omega(s)|$ small, when $|t- s|$ is not large. In other words, these paths tend to clump together on short time scales. 
However, for small $\eps>0$, the interaction gets more and more smeared out and the Polaron interaction, at least on an intuitive level, should behave like the {\it mean-field interaction} given by the measures 
\begin{equation}\label{Phat}
\widehat{\P}_T^{\mathrm{(mf)}}(\d \omega)=\frac 1 { Z_T^{\mathrm{(mf)}}}\, \exp\bigg\{\frac 1{2T}\int_{-T}^T\int_{-T}^T \d t \d s \,\frac 1{\big|\omega(t)-\omega(s)\big|}\bigg\} \,\d\P.
\end{equation}
The earlier result \eqref{Pekarfor} then justified
this intuition and underlined the confluent behavior (on a logarithmic scale) for the partition functions of these two models:
$$
g_0= \lim_{\eps\to 0}\lim_{T\to\infty}\frac 1T\log Z_{\eps,T}= \lim_{T\to\infty}\frac 1T\log Z_{T}^{\mathrm{(mf)}}.
$$
Thus, it is natural to guess that 
the asymptotic behavior  of the actual Polaron measures $\widehat\P_{\eps,T}$ (as $T\to\infty$ followed by $\eps\to 0$) should somehow be related to its mean-field counterpart $\lim_{T\to\infty}\widehat{\P}_T^{\mathrm{(mf)}}$,
whose  behavior was determined in \cite{BKM15} based on the theory developed in \cite{MV14} and its extension \cite{KM15}.
It was shown in (\cite{BKM15}) that the distribution 
$\widehat\P_T^{\mathrm{(mf)}}\, L_T^{-1}$ of the Brownian occupation measures $L_T=\frac 1 T\int_0^T\delta_{W_s} \,\d s $ under the mean-field model converges to the distribution 
of a random translation $[\psi_0^2\star \delta_X]\,\d z$ of $\psi_0^2 \, \d z$, with the random shift $X$ having a density $\psi_0 /\int \psi_0$. 
 Furthermore, it was also shown in \cite{BKM15} that the mean-field measures $\widehat\P_T^{\mathrm{(mf)}}$ themselves converge, as $T\to\infty$ towards
 a spatially inhomogeneous mixture of the stationary ergodic process driven by the SDE 
$$
\begin{aligned}
&\d X_t= \d W_t+ \bigg(\frac{\nabla \psi_y}{\psi_y}\bigg)(X_t)\,\d t \\
& X_0=0
\end{aligned}
$$
with $\psi_y$ being centered around $y$, with $y$ being distributed as $\psi_0/\int \psi_0$, and $\psi_0$ is centered around $0$. This result consequently led to a rigorous construction of the {\it{Pekar process}} whose heuristic definition was 
 set forth in \cite{Sp87}. In the context of the present article, the strong coupling limit $\widehat{\mathbb Q}^{\ssup \psi}=\lim_{\eps\to 0} \widehat \P_\eps= \lim_{\eps \to 0} \lim_{T\to\infty} \widehat\P_{\eps,T}$ of the 
 Polaron measures coincide with the increments of the stationary  Pekar process. Thus, in light of the above heuristic discussion, our main result also justifies the mean-field approximation of the Polaron problem 
 on the level of path measures. 

\section{Main results: Convergence of the Polaron measures to Pekar process in strong coupling}\label{sec-results}

Before we turn to the precise statements of our main results,
it is useful to set some notation which will be used throughout the sequel.

\subsection{Notation.}
We will denote by $\mathcal X=C(\R,\R^d)$ the space of all $\R^d$-valued continuous functions $\omega$ on the real line, which is equipped with the topology of uniform convergence of compact sets.
 Note that $\R^d$ acts as an additive group of  translations
on $\mathcal X$, and for any $v\in \R^d$, its action will be denote by $(\tau_v \omega)(\cdot)= \omega(v+\cdot)$. Let $\Mcal_1(\mathcal X)$ denote the space of probability measures equipped with the weak topology 
which is characterized by convergence of integrals against continuous and bounded functions on $\mathcal X$. Then $\Mcal_{\mathrm s} \subset \Mcal_1(\mathcal X)$ will stand for the space of invariant probability measures (i.e., 
$\mathbb Q\in \Mcal_{\mathrm s}$ if and only if $\mathbb Q=\tau_v \,\mathbb Q$ for all $v \in \R^d$), or the space of stationary processes taking values in $\R^d$. It is well-known that $\Mcal_{\mathrm s}$ is a convex set, and its extreme points, or the ergodic probability measures on $\mathcal X$ will be denoted by $\Mcal_{\mathrm s, \mathrm e}$. 

We will also set $\mathcal Y$ to be the space of $\R^d$-valued continuous trajectories $\omega(s,t)$ defined for $-\infty<s\leq t<\infty$ such that 
$$
\omega(s,t)+\omega(t,u)=\omega(s,u)\qquad\forall -\infty<s\leq t\leq u<\infty.
$$
The group $\{\tau_v\}_{v\in\R^d}$ also acts on $\mathcal Y$ by $(\tau_v\omega)(s,t)=\omega(s+v,t+v)$. We will denote by $\Mcal_{\mathrm{si}}$ to be the space of all probability measures on $\mathcal Y$ which are 
invariant under the above action, or the {\it{space of processes with stationary increments}}. Its extreme points, or the space of ergodic probability measures on $\mathcal Y$ will be denoted by $\Mcal_{\mathrm{si},\mathrm e}$.
Note that we have a canonical map $\Phi:\mathcal X\to \mathcal Y$ defined by 
\begin{equation}\label{Phi}
(\Phi\omega)(s,t)=\omega(t)-\omega(s)
\end{equation}
which also induces a map from $\Mcal_{\mathrm s}$ to $\Mcal_{\mathrm{si}}$. Throughout the rest of the article, $\P\in \Mcal_{\mathrm{si}}$ will stand for the law of three dimensional Brownian increments defined on the space $\mathcal Y$ equipped with the $\sigma$-algebra generated by the increments $\{\omega(t)-\omega(s)\}$.

For any $-\infty<a<b<\infty$ throughout this article we will denote by $\mathcal F_{[a,b]}$  the 
$\sigma$-algebra  generated by all increments $\{\omega(s)-\omega(r)\colon a \leq r < s \leq b\}$.
For any $\mathbb Q\in \Mcal_{\mathrm{si}}$, if $\mathbb Q_{t,\omega}$ denotes the regular conditional probability distribution  of $\mathbb Q$ given $\mathcal F_{(-\infty,t]}$, then 
$$
H_t(\mathbb Q| \P)= \E^{\mathbb Q}\bigg[h_{\mathcal F_{[0,t]}}(\mathbb Q_{t,\omega} \big | \P)\bigg]
$$
defines the entropy of the process $\mathbb Q$ with respect to $\P$ at a given time $t>0$. In the above display, for any two probability measures $\mu$ and $\nu$ on any $\sigma$-algebra of the form $\mathcal F=\mathcal F_{[a,b]}$ on $\mathcal Y$, we denoted by 
$$
h_{\mathcal F}(\mu|\nu)=\sup_f \bigg\{\int f\d\mu- \log\bigg(\int \e^f \d\nu\bigg)\bigg\}
$$ 
the relative entropy of the probability measure $\mu$ with respect to $\nu$ on $\mathcal F$, with the supremum above being taken over all 
continuous, bounded and $\mathcal F$-measurable functions. For our purposes,  
it is useful the collect some properties of $H_t(\mathbb Q| \P)$ which were deduced in (Part IV, \cite{DV75}):
\begin{lemma}\label{lemma-entropy}
Either $H_t(\mathbb Q| \P)\equiv\infty$ for all $t>0$, or, $H_t(\mathbb Q|\P)= t H(\mathbb Q|\P)$, where the map  
$H(\cdot| \P): \Mcal_{\mathrm{si}} \to [0,\infty]$, called the {\it{( specific) relative entropy of the process $\mathbb Q$ with respect to $\P$}}, satisfies the following properties: 
\begin{itemize}
\item $H(\cdot|\P)$ is convex and lower-semicontinuous in the usual weak topology. 
\item$H(\cdot|\P)$ is coercive (i.e., for any $a\geq 0$, 
the sub-level sets $\{\mathbb Q\colon H(\mathbb Q| \P)\leq a\}$ are weakly compact.
\item The map $\mathbb Q\mapsto H(\mathbb Q|\P)$ is linear. In particular, for any probability measure $\Gamma$ on $\Mcal_{\mathrm{si}}$, 
$$
H\bigg(\int \mathbb Q \, \Gamma(\d\mathbb Q)\bigg|\P\bigg)= \int H(\mathbb Q|\P) \,\Gamma(\d\mathbb Q).
$$
\end{itemize}
\end{lemma}

\subsection{Main results.}
We will now provide a precise description of our main results. For any fixed $\eps>0$ and $T>0$, let 
\begin{equation}\label{eq-Polaron-1}
\d\widehat\P_{\eps,T}=\frac 1 {Z_{\eps,T}} \exp\bigg[ \frac 1 2\int_{-T}^T\int_{-T}^T \frac{ \eps\e^{-\eps\,|t-s|}\,\d s \,\d t}{|\omega(t)-\omega(s)|} \bigg] \,\,\d\P_T
\end{equation}
denote the Polaron measure on the space $\mathcal Y$ equipped with the $\sigma$-algebra generated 
by the increments $\{\omega(t)-\omega(s): -\infty <s < t< \infty\}$. We remind the reader that $\P_T$ is the restriction of the law $\P$ of three dimensional Brownian increments to $[-T,T]$, 

We also recall that (\cite{DV83})
$$
g_0= \sup_{\heap{\psi\in H^1(\R^3)}{\|\psi\|_2=1}} 
\Bigg\{\int_{\R^3}\int_{\R^3}\d x\d y\,\frac {\psi^2(x) \psi^2(y)}{|x-y|} -\frac 12\big\|\nabla \psi\big\|_2^2\Bigg\},
$$
and the non-empty set $\mathfrak m=\{\psi_0^2\star \delta_x\colon x\in \R^3\}$ of the maximizing densities for $g_0$ consists of spatial translations of some $\psi_0\in H^1(\R^3)$ with $\int_{\R^3} \psi_0^2=1$ (\cite{L76}).
Let $\widehat{\mathbb Q}^{\ssup \psi}$ be the increments of the stationary ergodic diffusion process 
with generator 
$$
L_\psi= \frac 12 \Delta+ \bigg(\frac{\nabla\psi}{\psi}\bigg)\,\cdot\nabla
$$
for any maximizer $\psi\in \mathfrak m$. We stress that, while $\psi$ and the associated stationary process with generator $L_\psi$ need not be unique, the increments of this stationary process are uniquely defined. 

Here is the statement of our main result: 
\begin{theorem}[Convergence of the Polaron measure to the increments of the Pekar process in strong coupling]\label{thm2}
With $\widehat{\mathbb Q}^{\ssup \psi}$ being the increments of the stationary Pekar process,
$$
\lim_{\eps\to 0} \, \lim_{T\to\infty} \widehat\P_{\eps,T}= \widehat{\mathbb Q}^{\ssup \psi}.
$$
\end{theorem}

An ingredient for the proof Theorem \ref{thm2} is the following result, which  is based on a {\it{strong large deviation principle}} for the distribution of the empirical process 
of Brownian increments.

\begin{theorem}\label{thm1}
Fix $\eps>0$. Then 
\begin{itemize}
\item 
\begin{equation}\label{eq-1-thm1}
\begin{aligned}
\lim_{T\to\infty} \frac 1 T \log Z_{\eps,T}&= g(\eps)\\
&=\sup_{\mathbb Q\in  \Mcal_{\mathrm{si}}} \bigg[ \E^{\mathbb Q}\bigg\{\int_0^\infty \frac{\eps\e^{-\eps r}\,\d r}{|\omega(r)-\omega(0)|}\bigg\}- H(\mathbb Q|\P)\bigg].
\end{aligned}
\end{equation}
\item The supremum in \eqref{eq-1-thm1} is attained for some process $\mathbb Q_\eps\in \Mcal_{\mathrm{si}}$  such that $H(\mathbb Q_\eps|\P)<\infty$. In other words, 
the set $\mathfrak m_\eps$ of maximizers  of $g(\eps)$ is non-empty. 
\item Let $\widehat\P_\eps$ be the stationary limit of $\widehat \P_{\eps,T}$ as $T\to\infty$ (see Section \ref{sec-Review} for a precise definition and expression for $\widehat\P_\eps$). Then $\widehat\P_\eps\in\mathfrak m_\eps$.
\end{itemize}
\end{theorem}

\begin{remark}\label{rmk1-thm1}
As remarked earlier, a variational formula for the free energy $g(\eps)$ was first obtained in \cite{DV83} where the supremum in \eqref{eq-1-thm1} was taken over all stationary processes $\mathbb Q\in \Mcal_{\mathrm{s}}$.
This result was a consequence of a weak large deviation principle for the empirical process of Brownian motion. However, in this case, the supremum over $\mathbb Q\in \Mcal_{\mathrm{s}}$ may not be attained. 
This issue is instantaneously resolved if we exploit the underlying i.i.d. structure of the noise providing exponential tightness. 
The resulting free energy variational formula with the supremum over $\mathbb Q\in \Mcal_{\mathrm{si}}$ is coercive which gurantees existence of (at least one) maximizer. 
\end{remark}


\begin{remark}\label{rmk3-thm1}
The fact that for any fixed $\eps>0$, the Polaron measures $\{\widehat\P_{\eps,T}\}_T$ are tight is an immediate consequence of the underlying strong large deviation principle which proves Theorem \ref{thm1}. However, 
we crucially use the existence of the limit $\lim_{T\to\infty}\widehat\P_{\eps,T}=\widehat\P_\eps$ and the stationarity of $\widehat\P_\eps$ to conclude that $\widehat\P_\eps\in\mathfrak m_\eps$.
\end{remark}

The rest of the article is organized as follows. In Section \ref{sec-Review} we will shortly review most important assertions derived in \cite{MV18} regarding the 
$T\to\infty$-asymptotic behavior of the Polaron measures $\widehat\P_{\eps,T}$ for fixed $\eps>0$. In Section \ref{sec-abstract-thm} we will prove some general results concerning processes with stationary increments, 
and in Section \ref{sec-pf-thm1-thm2} we will prove Theorem \ref{thm1} Theorem \ref{thm2}. 

\subsection{Review: Identification of the Polaron measure as $T\to\infty$.}\label{sec-Review} 
For convenience, in this section we will be consistent with the notation in \cite{MV18} and we will work with the Polaron measure $\widehat\P_{\alpha,T}$ defined in \eqref{eq-Polaron-alpha}. 
Recall that the scaling relation \eqref{eq-alpha-eps} provides the link between the interactions in $\widehat\P_{\alpha,T}$ and $\widehat\P_{\eps,T}$ (defined in \eqref{eq-Polaron-1}) with the choice $\alpha=\frac 1 {\sqrt\eps}$.

The first crucial step in \cite{MV18}  is a representation of the Polaron measure $\widehat\P_{\alpha,T}$ for any $\alpha>0$ and $T>0$, 
as a mixture of Gaussian measures. 
Note that the Coulomb potential can be written as 
\begin{equation}\label{eq-Coulomb-Gaussian}
 \frac1{|x|}=\frac 1 {\Gamma(1/2)} \int_0^\infty \d u \,\e^{-u|x|^2} \,\frac 1 {\sqrt u}= c_0\int_0^\infty \e^{-\frac 12u^2|x|^2} \, \d u
\end{equation}
where $c_0=\sqrt{\frac{2}{\pi}}$. Then with ${\widehat\P}_{\alpha,T}=\frac 1 {Z_{\alpha,T}} \mathcal H_{\alpha,T}(\omega)\d\P$, we can expand the exponential weight 
$$
\mathcal H_{\alpha,T}(\omega)
=\exp\bigg\{\frac \alpha 2\int_{-T}^T\int_{-T}^T\frac{\e^{-|t-s|\d s \d t}}{|\omega(t)-\omega(s)|}\bigg\}
$$
 into a power series and invoke the above representation of the Coulomb potential to get 
\begin{equation}\label{eq-Polaron-01}
\begin{aligned}
\mathcal H_{\alpha,T}&=\sum_{n=0}^\infty \frac{\alpha^n}{n!} \bigg[\int\int_{-T\leq s \leq t\leq T} \frac{\e^{-|t-s|}\,\d t \, \d s}{|\omega(t)-\omega(s)|}\bigg]^n \\
&= \sum_{n=0}^\infty \frac {1}{n!} \prod_{i=1}^n \bigg[\bigg(\int\int_{-T\leq s_i \leq t_i\leq T} \big(\alpha\,\e^{-(t_i-s_i)}\,\d s_i\, \d t_i\big)\bigg)\,\,  \bigg( c_0 \int_0^\infty \, \d u_i \e^{-\frac 12 u_i^2 |\omega(t_i)-\omega(s_i)|^2}\bigg)\bigg].
\end{aligned}
\end{equation}
Note that, when properly normalized, $\mathcal H_{\alpha,T}$ is a mixture of (negative) exponentials of positive definite quadratic forms. 
Next, in the second display in \eqref{eq-Polaron-01}, we 
have a Poisson point process taking values on the space of finite intervals $[s,t]$ of $[-T,T]$  with intensity measure 
$$
\gamma_\alpha(\d s\,\d t)=\alpha \e^{-(t-s)}\d s\d t
$$
 on $-T\le s<t\le T$. Then it turns out that, for any $\alpha>0$ and $T>0$, 
we have a representation 
\begin{equation}\label{eq-Polaron-03}
\widehat\P_{\alpha,T}(\cdot)=\int_{\widehat{\mathcal Y}} \mathbf P_{\hat\xi,\hat u} (\cdot) \,{\widehat  \Theta}_{\alpha,T}(\d\hat\xi\,\d \hat u).
\end{equation}
of the Polaron measure as a superposition of Gaussian measures $\mathbf P_{\hat\xi,\hat u}$ indexed by $(\hat\xi,\hat u)\in \widehat{\mathcal Y}$ with $\widehat{\mathcal Y}$ being the space
of all collections of (possibly overlapping) intervals $\widehat\xi=\{[s_1,t_1],\dots,[s_n,t_n]\}_{n\geq 0}$ and strings $\widehat u\in (0,\infty)^n$, while
the ``mixing measure" ${\widehat \Theta}_{\alpha,T}$ is a suitably defined probability measure on the space $\widehat{\mathcal Y}$.
As an immediate corollary, we obtain 
that for any fixed $\alpha>0$ and $T>0$, the variance of any linear functional on the space of increments with respect to $\widehat\P_{\alpha,T}$ is dominated by 
the variance of the same with respect to the restriction $\P_T$ of  $\P$ to $[-T,T]$. The details of the Gaussian representation \eqref{eq-Polaron-03} can be found in Theorem 2.1 in \cite{MV18}.

Then the limiting behavior $\lim_{T\to\infty} \widehat\P_{\alpha,T}$ of the Polaron (and hence, the central limit theorem for the rescaled increment process) 
follows once we prove a law of large numbers for the mixing measure ${\widehat\Theta}_{\alpha,T}$. This measure is defined as a tilted probability measure w.r.t. the law of the aforementioned Poisson process
with intensity $\gamma_{\alpha,T}$. Note that, the union of any collection of intervals $\{[s_i,t_i]\}$, which is a typical 
realization of this Poisson process,
need not be connected. In fact, the union is a  disjoint union of connected intervals, with gaps in between, starting and ending with gaps $[-T,\min\{s_i\}]$ and $[\max\{t_i\},T]$.
It is useful to interpret this Poisson process as a birth-death process along with some extra information (with ``birth of a particle at time $s$ and  the 
same particle dying at time $t$")  that links each birth with the corresponding death.
The birth rate is 
$$
b_{\alpha,T}(s)=\alpha(1-\e^{-(T-s)})
$$
and the death rate is 
$$
d_{\alpha,T}(s)=\frac 1 {1-\e^{-(T-s)}}
$$
and both rates are computed from the intensity measure $\gamma_{\alpha,T}$. 
As $T\to\infty$, the birth and death rates converge to constant birth rate $\alpha>0$ and death rate $1$, and 
we imagine the infinite time interval $(-\infty, \infty)$ to split into an alternating sequence 
of ``gaps" and ``clusters" of overlapping intervals. The gaps are called {\it{dormant periods}} (when no individual is alive and the population size is zero) and will be denoted by $\xi^\prime$, while 
each {\it{cluster}} or an {\it{active period}} is a collection $\xi=\{[s_i,t_i]\}_{i=1}^{n(\xi)}$ of overlapping intervals with the union $\mathcal J(\xi)=\cup_{i=1}^{n(\xi)} [s_i,t_i]$ being a connected interval without any gap.
Note that, inception times of both dormant and active periods possess the {\it{regeneration property}}, i.e., all prior information is lost and there is a fresh start. 
Also, on any dormant period $\xi^\prime$, the aforementioned Gaussian measure $\mathbf P_{\xi^\prime,u}\equiv \P$ corresponds only to the law of Brownian increments, 
and independence of increments on disjoint intervals (i.e., alternating sequence of dormant and active periods) leads to a
``product structure"  for  the mixing measure ${\widehat\Theta}_{\alpha,T}$. Indeed, if $\Pi_\alpha$ denotes the law of the above birth death process in a single active period 
with constant birth rate $\alpha>0$ and death rate $1$,
then a crucial estimate (Theorem 3.1, \cite{MV18}) shows that for any $\alpha>0$, there exists  
$\lambda_0(\alpha)>0$ such that, for $\lambda>\lambda_0(\alpha)$
$$
q(\lambda)= \E^{\Pi_{\alpha}\otimes\mu_\alpha} \big[\exp\{-\lambda[|\mathcal J(\xi)| +|\xi^\prime|]\}\mathbf F(\xi)\big]<\infty,
$$
where $\mu_\alpha$ is exponential distribution of parameter $\alpha$ and 
$$
\begin{aligned}
&\mathbf F(\xi)=\bigg(\sqrt{\frac 2\pi}\bigg)^{n(\xi)} \int_{(0,\infty)^{n(\xi)}}\mathbf \Phi(\xi,\bar u)\,\,\d \bar u \end{aligned}
$$
and $\mathbf \Phi(\xi,\bar u)= \E^\P[\exp\{-\frac 12 \sum_{i=1}^{n(\xi)} u_i^2 |\omega(t_i)-\omega(s_i)|^2\}]$ is the normalizing constant for the Gaussian measure $\mathbf P_{\xi,\bar u}$ in one active period $(\xi,\bar u)$.
It turns out that, there exists $\alpha_0,\alpha_1\in(0,\infty)$ such that when $\alpha\in(0,\alpha_0)$ or $\alpha\in(\alpha_1,\infty)$, 
there exists $\lambda=\lambda(\alpha)$ such that $q(\lambda)=1$. 

For such a choice of $\lambda$ (which forces $q(\lambda)=1$), the underlying renewal structure of the active and dormant periods imply that the mixing measure $\widehat\Theta_{\alpha,T}$ of the Polaron 
$\widehat\P_{\alpha,T}$ converges as $T\to\infty$ to the 
stationary version $\widehat{\mathbb Q}_\alpha$ on $\R$ obtained by alternating the limiting mixing measure on each active period $\xi$ defined as
$$
\widehat{\Pi}_\alpha(\d\xi\,\d \bar u)= \bigg(\frac {\alpha}{\lambda+\alpha}\bigg)\bigg[  \e^{-\lambda |{\mathcal J}(\xi)|}\,\,\bigg(\frac 2\pi\bigg)^{\frac {n(\xi)} 2}\,\big[\mathbf \Phi(\xi, \bar u)\,\, \d \bar u\big]\bigg] \,\,\Pi_\alpha(\d\xi),
$$ 
and  as the tilted exponential distribution 
$$
\widehat\mu_\alpha(\d\xi^\prime)=\bigg(\frac{\alpha+\lambda}\alpha\bigg)\,\, \e^{-\lambda|\xi^\prime|} \,\,\mu_\alpha(\d\xi^\prime)
$$
on each dormant period $\xi^\prime$ with expected waiting time  $(\lambda+\alpha)^{-1}$. 

Thus, given the Gaussian representation \eqref{eq-Polaron-03}, 
 the Polaron measure $\widehat\P_{\alpha,T}$ then converges  as $T\to\infty$, in total variation on finite intervals in $(-\infty,\infty)$,  to
$$
\widehat\P_\alpha(\cdot)=\int \mathbf P_{\hat\xi, \hat u}(\cdot)\, \,\widehat {\mathbb Q}_{\alpha}(\d\hat\xi\,\d\hat u),$$
where on the right hand side, $\mathbf P_{\hat\xi, \hat u}$ is the product of the Gaussian measures $\mathbf P_{\xi, \bar u}$ on the active intervals and law $\P$ of Brownian increments on dormant intervals, and 
the integral above is taken over the space of all active intervals (with $\bar u=(u_i)_{i=1}^{n(\xi)}$ and $u_i$'s being attached to each birth with the corresponding death) as well as dormant intervals. See Theorem 4.1 in \cite{MV18} for details.

The central limit theorem for the rescaled increment process $(2T)^{-1/2}\, [\omega(T)-\omega(-T)]$ under $\widehat\P_{\alpha,T}$ as $T\to\infty$ also follows 
 readily. It turns out that the variance in each dormant period $\xi^\prime$ is just the expected length $(\alpha+\lambda)^{-1}$ of the empty period, 
 and 
 the resulting central limit 
 covariance matrix is $\sigma^2(\alpha) I$, where for any unit vector $v\in\R^3$ and any active period $\xi=[0,\sigma^\star]$, 
 $$
  \sigma^2(\alpha)= \frac{(\alpha+\lambda)^{-1}+ \E^{\widehat{\Pi}_\alpha}\big[\E^{\mathbf P_{\xi,\bar u}}[\langle v,\omega(\sigma^\star)-\omega(0)\rangle^2]\big]}{(\alpha+\lambda)^{-1}+  \E^{\widehat{\Pi}_\alpha}[\sigma^\star]}\in (0,1].
 $$
 We refer to Theorem 4.2 in \cite{MV18} for further details.


\section{Some results for processes with stationary increments.}\label{sec-abstract-thm}

 In order to derive Theorem \ref{thm2}, we will need to prove first an abstract result. We remind the reader that $\mathcal X$ denotes the space of all continuous functions from $\R$ to $\R^d$. If we denote by 
$\mathcal X_0 =\{\xi\in \mathcal X\colon \xi(0)=0\}$, then there is a one-to-one correspondence between $\mathcal X$ and $\mathcal X_0\otimes \R^d$ (i.e., given $\omega \in \mathcal X$, we can set $\xi(t)= \omega(t)-\omega(0)$, 
and $a=\omega(0)$ so that $(\xi,a)\in \mathcal X_0\otimes \R^d$). Note that the action of $\{\tau_v\}_{v\in\R}$ with $(\tau_t\omega)(s)= \omega(s+t)$ now acts on $\mathcal X_0\otimes \R^d$ as 
$\tau_t (\xi,a) =(\xi^\prime, a^\prime)$ where $\xi^\prime(t)= \xi(s+t)- \xi(t)$ and $a^\prime= a+ \xi(t)$. 

We need to define a suitable topology on finite measures on the space $\mathcal X_0\otimes \R^d$. Note that, weak convergence of a sequence of finite measures $(\mu_n)_n$ on a Polish space is characterized by convergence of integrals $\int f \d\mu_n$ for continuous and bounded functions $f$, while vague convergence of $(\mu_n)_n$ requires convergence of the integrals w.r.t. continuous functions with compact support, or continuous functions vanishing at infinity. We remark that in the vague topology of measures, a sequence of finite measures always finds a convergent subsequence, while for weakly convergent subsequences are determined by uniform tightness. 

We say that, a sequence of finite measures $(\mu_n)_n$ on $\mathcal X_0\otimes \R^d$ converges {\it{wea-guely}} to $\mu$  if and only if 
\begin{equation}\label{def-weague}
\int F(\xi,a) \d\mu_n \to \int F(\xi,a) \d\mu
\end{equation}
for all continuous bounded function $F:\mathcal X_0\otimes\R^d \to \R$ such that 
$$
\lim_{a\to\infty} \sup_{\xi\in \mathcal X_0} F(\xi,a)=0. 
$$
We remark that any sequence of measures $(\mathbf P_n)_n$ on $\mathcal X_0\otimes \R^d$, whose projection on $\mathcal X_0$ is uniformly tight, has a wea-guely convergent subsequence. 
Here is the first result of this section. 

\begin{theorem}\label{thm3}
Let $\mathbf P \in \Mcal_{\mathrm{si,e}}(\mathcal X)$ be an almost surely continuous ergodic process with stationary increments with values in $\R^d$. Then either, 
\begin{equation}\label{eq1-thm3}
\lim_{\eps\to 0} \eps \,\, \E^{\mathbf P} \bigg[\int_0^\infty  \e^{-\eps t}\,\,  V\big(\omega(t)-\omega(0)\big) \,\ \d t\bigg]=0,
\end{equation}
for all continuous functions $V:\R^d \to \R$ with compact support, or there is a stationary process $\mathbb Q\in \Mcal_{\mathrm{s}}(\mathcal X)$ with $\mathbf P$ being the distribution of its increments. 
\end{theorem}
\begin{proof}
Note that we can assume $\mathbf P(\omega(0)=0)=1$ and consider $\mathbf P$ as a measure on $\mathcal X_0\times \R^d$. Since $\{\tau_t\}$ acts on the space of finite measures on $\mathcal X_0\otimes \R^d$, 
we define the average 
$$
{\mathbf P}_\eps= \eps \int_0^\infty \e^{-\eps t} \, \tau_t \mathbf P \,\, \d t.
$$
Since $\mathbf P$ has stationary increments, for any $\eps>0$ and $A\subset \mathcal X_0$, we have $\int_0^\infty \eps \, \e^{-\eps t }\,\, \mathbf P(\tau_t\omega\in \in A) \,\d t = \int_0^\infty \eps \,\e^{-\eps t}\, \mathbf P(A)\,\,\d t=\mathbf P(A)$, and the projection of $\mathbf P_\eps$ on $\mathcal X_0$ is still $\mathbf P$. Then there is a subsequence $\mathbf P_n:=\mathbf P_{\eps_n}$ which converges wea-guely to a finite measure $\mathbb Q$ on $\mathcal X_0\otimes \R^d$. 

Let us now assume that \eqref{eq1-thm3} is not true. This implies that for some continuous function $V\geq 0$ with compact support, 
$$
\limsup_{\eps\to 0} \int V(a) \d\mathbf P_\eps>0.
$$
Since $\mathbf P_n \to \mathbb Q$ wea-guely, then $\int V(a) \mathbb Q>0$, forcing 
$$
\mathbb Q(\mathcal X_0\otimes \R^d)=m>0.
$$
We want to assert that, $\tau_t \mathbb Q= \mathbb Q$ for all $t$. This will imply that $\mathbb Q\in \Mcal_{\mathrm{s}}(\mathcal X_0\otimes \R^d)$ is stationary, and its projection of $\mathcal X_0$ will have stationary increments which will be dominated by $\mathbf P$. But the ergodicity of $\mathbf P$ implies that $\mathbb Q= m \mathbf P$, and $\mathbb Q/m$ is a stationary process with its increments distributed as $\mathbf P$, which will prove the claim. 

To check $\tau_t \mathbb Q= \mathbb Q$, we will verify that 
$$
\int F\big(\tau_t \xi, a+ \xi(t)\big) \d\mathbb Q= \int F(\xi,a)\d\mathbb Q.
$$
By construction, we note that $\|\mathbf P_\eps- \tau_t\mathbf P_\eps\|\to 0$ as $\eps\to 0$. Indeed, for any continuous and bounded function $F$ on $\mathcal X_0\otimes \R^d$, 
$$
\begin{aligned}
&\bigg| \int_0^\infty F\big(\tau_s(\xi,a)\big) \,\,\eps \e^{-\eps s}\,\, \d s - \int_0^\infty F\big(\tau_{s+t}(\xi,a)\big) \,\,\eps \e^{-\eps s}\,\, \d s \bigg| \\
&\bigg| \int_0^t F\big(\tau_s(\xi,a)\big) \,\,\eps \e^{-\eps s}\,\, \d s - \int_t^\infty F\big(\tau_{s}(\xi,a)\big) \,\,\big[\eps \e^{-\eps (s-t)}- \eps \e^{-\eps s}\,\, \d s \bigg|  \\
&\leq \eps t \|F\|_\infty+ \|F\|_\infty [1-\e^{-\eps t} ],
\end{aligned}
$$
implying that $\|\mathbf P_\eps- \tau_t\mathbf P_\eps\|\to 0$. Now, since $\mathbf P_n=\mathbf P_{\eps_n}$ converges wea-guely to $\mathbb Q$, it suffices to show that $\tau_t \mathbf P_n \to \tau_t \mathbb Q$ wea-guely, or
$$
\int F\big(\tau_t \xi, a+ \xi(t)\big) \d\mathbf P_n \to F\big(\tau_t \xi, a+ \xi(t)\big) \d\mathbb Q.
$$
Note that by construction, the distribution of $\xi(t)$ is that of the increment $\omega(t)-\omega(0)$. Hence, given $\eps>0$, there is a function $g=g_\eps:\R^d\to [0,1]$ with compact support, 
such that 
$$
\int \big(1-g(\xi(t))\big) \d\mathbf P_n <\eps \qquad \mbox{and}\,\,\int \big(1-g(\xi(t))\big) \d\mathbb Q <\eps.
$$
If we now replace $F\big(\tau_t \xi, a+ \xi(t)\big)$ by $\big(1-g(\xi(t))\big) F\big(\tau_t \xi, a+ \xi(t)\big)$, since $\xi(t)$ is restricted to a compact set in $\R^d$, 
$$
\lim_{a\to\infty} \sup_{\xi\in \mathcal X_0} F(\xi,a)=0. 
$$
and we are done. 

\end{proof}

Now we will need a technical fact which enables us  to derive Theorem \ref{thm3} for the interaction $V$ replaced by the singular Coulomb potential in $\R^3$.

\begin{lemma}\label{lemma-Coulomb}
For any $\eta>0$, let 
\begin{equation}\label{Y_eps}
\begin{aligned}
&V(x)=\frac 1 {|x|} \qquad \,\,\, V_\eta(x)= \frac 1 {(\eta^2+|x|^2)^{1/2}} \quad \mbox{and}\\
&Y_\eta(x)= (V-V_\eta)(x)
\end{aligned}
\end{equation}
Then, 
$$
\lim_{\eta\to 0} \,\,\sup_{\mathbb Q\in \Mcal_{\mathrm{si}}} \,\,\limsup_{\eps\to 0} \,\,\eps\,\,\E^{\mathbb Q} \bigg[\int_0^\infty \d t \,\, \e^{-\eps t}\,\,Y_\eta\big(\omega(t)- \omega(0)\big)\bigg]=0.
$$
where the supremum above is taken over all $\mathbb Q\in \Mcal_{\mathrm{si}}$ such that $H(\mathbb Q |\P)\leq C$ for some $C<\infty$.
\end{lemma}
\begin{proof}
We set 
$$
\begin{aligned}
\Psi(\omega)= \int_0^\infty \d t \,\, \e^{-\eps t}\,\,Y_\eta\big(\omega(t)- \omega(0)\big) &=\sum_{n=0}^\infty \int_n^{n+1} \d t \,\, \e^{-\eps t}\,\,Y_\eta\big(\omega(t)- \omega(0)\big) \\
&\leq \sum_{n=0}^\infty \e^{-\eps n} \Psi_n(\omega).
\end{aligned}
$$
where $\Psi_n(\omega)= \int_n^{n+1} \d t \,\, \,\,Y_\eta\big(\omega(t)- \omega(0)\big)$, which is $\mathcal F_{[n,n+1]}$ measurable. 
We now invoke the relative entropy estimate which asserts that for any  function $\Psi: \mathcal Y\to \R$ which is bounded and $\mathcal F_{[a,b]}$-measurable, for any $\beta>0$, 
\begin{equation}\label{eq-entropy-bound}
\E^{\mathbb Q} [ \Psi ] \leq \frac 1 \beta \bigg[(b-a)  {H(\mathbb Q| \P)} + \log \,\E^\P\big[\e^{\beta\Psi}\big]\bigg],
\end{equation}
Therefore, 
$$
\eps \E^{\mathbb Q}[\Psi] \leq \eps \sum_{n=0}^\infty \e^{-\eps n} \bigg[ \frac 1 \beta H(\mathbb Q|\P)+ \frac 1 \beta \log \E^\P\big[\e^{\beta\Psi_n}\big]\bigg]
$$
By our assumption $H(\mathbb Q|\P)\leq C$, if we now send $\eps\to 0$, followed by $\beta\to\infty$, it remains to show that, for any $\beta>0$,
\begin{equation}\label{eq-superexp}
\lim_{\eta\to 0} \sup_n  \log \E^\P\big[ \e^{\beta \Psi_n}\big] =0.
\end{equation}
Note that 
$$
Y_\eta(x)= \eta^{-1} \,\varphi(\frac x\eta) \quad\mbox{where}\,\, \varphi(x)= \frac 1 {|x|}\,\frac 1{\sqrt{1+|x|^2}} \,\, \frac 1 {(|x|+ \sqrt{1+|x|^2}}.
$$
In particular, for some constant $C<\infty$, 
\begin{equation}\label{eq-Y}
Y_\eta(x)\leq \frac{C\sqrt\eta}{|x|^{3/2}}.
\end{equation}
It is well-known that, for any function $\widetilde V\geq 0$  and Markov process $\{\P_x\}$, if
\begin{equation}\label{check1}
\sup_{x\in \R^d}\E^{\P_x}\bigg[\int_0^1 \d s \,\, \widetilde V(\omega(s)) \bigg] \leq \theta <1,
\end{equation}
then 
$$
\sup_{x\in \R^d}\E^{\P_x}\bigg[\exp\bigg\{\int_0^1 \d s \,\, \widetilde V(\omega(s))\bigg\} \bigg] \leq \frac\theta{1-\theta}<\infty,
$$
Thus, in view of the above fact, if we now apply \eqref{eq-Y}, then the desired estimate \eqref{eq-superexp} follows once we show \eqref{check1}
with $\P_x$ denoting the law of three dimensional Brownian motion starting at $x\in\R^3$. If we choose $\eta>0$ small enough, then \eqref{check1} follows once we show that
$$
\begin{aligned}
\sup_{x\in \R^3} \E^{\P_x}\bigg\{\int_0^1 \frac{\d s}{|\omega_s|^\frac{3}{2}}\bigg\}
=\sup_{x\in \R^3} \d y\int_0^1 \d \sigma\int_{\R^3}\frac{1}{|y|^\frac{3}{2}} \frac{1}{(2\pi \sigma)^\frac{3}{2} }\exp\bigg\{-\frac{(y-x)^2}{2\sigma}\bigg\}
<\infty.
\end{aligned}
$$
One can see that  
\begin{equation}\label{eq-sup-x}
\sup_{x\in \R^3} \int_{\R^3}\d y\frac{1}{|y|^\frac{3}{2}} \frac{1}{(2\pi \sigma)^\frac{3}{2} }\exp\bigg\{-\frac{(y-x)^2}{2\sigma}\bigg\}
\end{equation}
is attained at $x=0$ because we  can rewrite the integral as 
$$ c\int_{\R^3} \exp\bigg\{-\frac{\sigma|y|^2}{2}+i\langle x,y\rangle\bigg\}\frac{1}{|y|^\frac{3}{2}}\d y,
$$ 
where $c>0$ is a constant. When  $x=0$ the integral reduces to $\int_0^1 \sigma^{-3/4} \,\,\d \sigma$ which is finite.
\end{proof}

Then we combine Theorem \ref{thm3} and Lemma \ref{lemma-Coulomb} to deduce

\begin{cor}\label{cor-thm3}
Let ${\mathbb Q} \in \Mcal_{\mathrm{si,e}}(\mathcal X)$ be an almost surely continuous ergodic process with stationary increments with values in $\R^3$ such that $H({\mathbb Q} | \P)<\infty$. 
Then either, 
\begin{equation}\label{eq1-thm3}
\lim_{\eps\to 0} \eps \,\, \E^{\mathbb Q} \bigg[\int_0^\infty  \e^{-\eps t}\,\,  \frac 1 {|\omega(t)-\omega(0)|} \,\ \d t\bigg]=0,
\end{equation}
or there is a stationary process $\mathbb Q^\prime\in \Mcal_{\mathrm{s}}(\mathcal X)$ with ${\mathbb Q}$ being the distribution of its increments. 
\end{cor}
\qed

\section{Proof of Theorem \ref{thm1} and Theorem \ref{thm2}}\label{sec-pf-thm1-thm2}

\subsection{Proof of Theorem \ref{thm1}.}

To prove Theorem \ref{thm1}, we will first provide a strong large deviation principle for the distribution of the empirical process $R_T$ of increments under $\P$. 
For any $A\subset \mathcal X$, we denote by 
\begin{equation}\label{Rdef}
R_T(A,\omega)=\frac 1 T\int_0^T \1_A\big\{\tau_s(\omega_T(\cdot)- \omega_T(0)\big\} \,\, \d s,
\end{equation}
which is the empirical process of Brownian increments. Here $\omega_T$ is the ``periodization" defined by
$$
\omega_T(s)= 
\begin{cases}
\omega(s) \qquad\qquad\qquad\mbox{ if } 0\leq s \leq T,
\\
n \omega(T)+ \omega(r)\qquad \,\mbox{if} \,\,s=nT+r \quad r,n\in\N, 0\leq r<T.
\end{cases}
$$
Note that for any $\omega\in \mathcal X$, 
$$
R_T(\omega,\cdot)\in \Mcal_{\mathrm{si}}
$$
 and thus we have an induced probability measure $\mathbb Q_T= \P \, R_T^{-1}$ on $\Mcal_{\mathrm{si}}$ via the above map $\mathcal X\to\Mcal_{\mathrm{si}}$.
 \begin{lemma}\label{lemma-1-pf-thm1}
The family $\mathbb Q_{T}= \P \,R_T^{-1}$  satisfies a strong large deviation principle as $T\to\infty$
in the space of probability measures on $\Mcal_{\mathrm{si}}$ with rate function $H(\cdot|\P)$ with compact level sets. In over words, 
\begin{equation}\label{eq-ldp-statement}
\begin{aligned}
\liminf_{T\to\infty}\frac 1 T \log \mathbb Q_T(G) \geq -\inf_{\mathbb Q\in G} H(\mathbb Q|\P) \qquad \forall\,\,\,G\subset \Mcal_{\mathrm{si}}\,\,\,\mbox{open}\\
\limsup_{T\to\infty}\frac 1 T \log \mathbb Q_T(F) \leq -\inf_{\mathbb Q\in F} H(\mathbb Q|\P) \qquad \forall\,\,\,F\subset \Mcal_{\mathrm{si}}\,\,\,\mbox{closed}
\end{aligned}
\end{equation}
\end{lemma}
\begin{proof}
The proof of the lower bound for all open sets $G\subset \Mcal_{\mathrm{si}}$ and the upper bound for all compact sets $K\subset \Mcal_{\mathrm{si}}$ in \eqref{eq-ldp-statement} 
follows directly from the arguments of (Part IV, \cite{DV75}) modulo slight changes, and the details are omitted.
To strengthen the upper bound to all closed sets $C\subset \Mcal_{\mathrm{si}}$, 
it suffices to show exponential tightness for the distributions $\mathbb Q_T$ that requires that for any $\ell>0$, existence of a compact set $K_\ell\subset\Mcal_{\mathrm{si}}$
so that, 
\begin{equation}\label{eq-1-pf-thm1}
\limsup_{T\to\infty} \frac 1 T\log \mathbb Q_T[K_\ell^{\mathrm c}] \leq -\ell.
\end{equation}
To prove the above claim, for $i=1,2,3$, if we set
$$
\|\omega_i^\star\|= \sup_{0\leq s \leq t\leq 1} \frac{|\omega_i(s)-\omega_i(t)|}{|s-t|^{1/4}}, 
$$
then by Borell's inequality, for  some constants $C_1, C_2>0$, 
$$
\P\big[\|\omega_i^\star\| > \lambda\big] \leq C_1 \exp\bigg[-\frac {\lambda^2}{2C_2}\bigg],
$$
and consequently, $\E^{\P}[\e^{\|\omega_i^\star\|}]<\infty$. Then 
$$
\limsup_{T\to\infty} \frac 1 T \log \E^{\mathbb Q_T}\big[\e^{T \sum_{i=1}^3 \|\omega_i^\star\|}\big] <\infty,
$$ 
and the desired exponential tightness \eqref{eq-1-pf-thm1} follows readily, proving the requisite upper bound in \eqref{eq-ldp-statement} for all closed sets. 
\end{proof}

We will now derive 
\begin{lemma}\label{lemma-2-pf-thm1}
Fix any $\eps>0$. Then the distributions 
$$
\widehat{\mathbb Q}_{\eps,T}=\widehat\P_{\eps,T} \,R_T^{-1}
$$
 of the empirical process of increments under the Polaron measure satisfies a strong large deviation principle in the space of probability measures on $\widetilde\Mcal$ with rate function 
$$
J_\eps(\mathbb Q)= g(\eps)- \bigg[\int F_\eps \mathbb \d \mathbb Q- H(\mathbb Q|\P)\bigg]
$$
and 
\begin{equation}\label{def-F}
F_\eps(\omega)=  \eps \int_0^\infty \frac{\eps \e^{-\eps r}\,\d r}{|\omega(r)-\omega(0)|}.
\end{equation}
\end{lemma}
\begin{proof}
Let us set
 $$
 F_{\eps,\eta}(\omega)= \eps \int_0^\infty \eps \e^{-\eps r}\, V_\eta\big(\omega(r)-\omega(0)\big)\,\d r,
 $$ 
with $V_\eta$ being the truncated Coulomb potential defined in \eqref{Y_eps}. Then by definition of the empirical process $R_T(\omega,\cdot)$, 
 $$
 \begin{aligned}
 T\int F_{\eps,\eta}(\omega^\prime) \, R_T(\omega,\d\omega^\prime)= \int_0^T\d t\bigg[\bigg\{\int_0^{T-t} \eps\e^{-\eps s} \,\,&V_\eta(\omega(t+s)-\omega(t))\,\,\d s\\
 &+ \int_{T-t}^{\infty} \eps\e^{-\eps s} \,\,V_\eta(\omega_T(t+s)-\omega_T(t))\,\,\d s \bigg\}\bigg].
 \end{aligned}
 $$
 If 
 $$
G_{\eps,T,\eta}(\omega)= \exp\bigg\{\int_0^T\int_0^T \eps \e^{-\eps |t-s|}\,\, V_\eta\big(\omega(t)-\omega(s)\big)\bigg\},
 $$ 
 then for any set $A\subset \Mcal_{\mathrm{si}}$, 
\begin{equation}\label{eq-estimate}
\E^{\P}\big[G_{\eps,T,\eta}(\omega)\,\,\1_A\big] \leq \E^{\P}\bigg[\exp\bigg\{ T\int F_\eta(\omega^\prime) \, R_T(\omega,\d\omega^\prime)\bigg\}\bigg]\,\,\e^{\eps/\eta}.
\end{equation}
Now we can combine lower and upper bounds from Lemma \ref{lemma-1-pf-thm1} with Varadhan's lemma applied to
the bounded continuous function $V_\eta$. A argument similar to the proof of Lemma \ref{lemma-Coulomb} implies that, for any $\beta>0$, 
$$
\limsup_{\eta\to 0} \,\limsup_{T\to\infty} \frac 1 T\log \E^{\P}\bigg[\exp\bigg\{\beta \,\, \eps\int_{-T}^T\int_{-T}^T \,\, \d s \d t\,\, \e^{-\eps |t-s|} \,\, Y_\eta\big(\omega(s)-\omega(t)\big)\bigg\}\bigg]=0.
$$
which then concludes the proof of Lemma \ref{lemma-2-pf-thm1}.
\end{proof}

\begin{lemma}\label{lemma-3-pf-thm1}
Fix any $\eps>0$ and let 
$$
\mathfrak m_\eps=\bigg\{\mathbb Q\in \Mcal_{\mathrm{si}}\colon \int F_\eps \d\mathbb Q- H(\mathbb Q| \P)= g(\eps)\bigg\}
$$
denote the set of processes with stationary increments that maximize the variational problem $g(\eps)$ defined in \eqref{eq-1-thm1}. Then, 
$\mathfrak m_\eps\neq\emptyset$. 
\end{lemma}
\begin{proof}
Since $\P$ has independent increments, if we  first show that for a maximizing sequence $\{\mathbb Q_n\}\subset \Mcal_{\mathrm{si}}$ for $g(\eps)$, 
\begin{equation}\label{entropy-uniform}
\sup_n H(\mathbb Q_n |\P) <\infty, 
\end{equation}
the above uniform bound then will imply that $\{\mathbb Q_n\}$ is tight, which, then combined with the 
lower semicontinuity of the map $\mathbb Q\mapsto H(\mathbb Q|\P)$, will force any limit point of $\mathbb Q_n$ to be also be a maximizer. 

In order to prove the uniform estimate \eqref{entropy-uniform}, we can take a maximizing sequence $\mathbb Q_n$ so that
\begin{equation}\label{estimate-1}
 \E^{\mathbb Q_n} \bigg [\int_0^\infty \frac{\eps \e^{-\eps r}\, \d r}{|\omega(r)-\omega(0)|}\bigg] \geq H(\mathbb Q_n|\P). 
 \end{equation} 
Note that, by Lemma \ref{lemma-2-pf-thm1}, we also have, for some $C>1$, 
$$
C \E^{\mathbb Q_n} \bigg [\int_0^\infty \frac{\eps \e^{-\eps r}\, \d r}{|\omega(r)-\omega(0)|}\bigg] H(\mathbb Q_n |\P) \leq g_C(\eps) <\infty,
$$
which combined with \eqref{estimate-1} implies
$(C-1) H(\mathbb Q_n |\P) \leq g_C(\eps)$ proving the desired uniform bound \eqref{entropy-uniform}.
 \end{proof}

 
 The following lemma will finish the proof of Theorem \ref{thm1}.
\begin{lemma}\label{lemma-4-pf-thm1}
$\lim_{T\to\infty} \widehat\P_{\eps,T}=\widehat\P_\eps\in\mathfrak m_\eps$, i.e., the limiting  process $\widehat\P_\eps$ with stationary increments is a maximizer for 
the variational formula $g(\eps)$. 
\end{lemma}
\begin{proof}
From the large deviation upper bound from Lemma \ref{lemma-2-pf-thm1}  it follows that if $U(\mathfrak m_\eps)$ is any open neighborhood of the maximizing set $\mathfrak m_\eps$ in the weak topology, then   
$$
\begin{aligned}
&\limsup_{T\to\infty} \frac 1 T\log \widehat{\P}_{\eps,T}[R_T\in U(\mathfrak m_\eps)^c] \\
&=-g(\eps) + \sup_{\mathbb Q\notin U(\mathfrak m_\eps)}\bigg[\int F\d\mathbb Q- H(\mathbb Q|\P)\bigg] \\
&<0.
\end{aligned}
$$
We recall that the rate function $\mathbb Q\mapsto H(\mathbb Q| \P)$ has compact sub-level sets. Then it follows that the distributions $\{\widehat\P_{\eps,T} \, R_T^{-1}\}_T$ are tight and any limit point is concentrated in $\mathfrak m_\eps$. Finally,  as shown in \cite{MV18}, $\widehat \P_\eps=\lim_{T\to\infty}\widehat \P_{\eps,T}$ is stationary, implying that
$\widehat \P_\eps \in \mathfrak m_\eps$. 
\end{proof}

\subsection{Strong coupling limit: Proof of Theorem \ref{thm2}.}

Recall that $\widehat \P_\eps \in \mathfrak m_\eps$, and 
\begin{equation}\label{eq-DV-free-energy}
\lim_{\eps\to 0} g(\eps)=g_0= \sup_{\mu\in \Mcal_1(\R^3)} \bigg[\int\int \frac{\mu(\d x) \mu(\d y)}{|x-y|}- I(\mu)\bigg]>0.
\end{equation}
with $I(\mu)$ being the classical large deviation rate function, i.e., $I(\mu)=\big\|\nabla\sqrt{\frac{\d\mu}{\d x}}\big\|_2^2$ if $\mu$ has a Lebesgue density in $H^1(\R^3)$, else $I(\mu)=\infty$. 
We also note that the minimizer of the variational formula 
\begin{equation}\label{eq-contr-pr}
\inf_{\mathbb Q\in \Mcal_{\mathrm{si}}}\,\,\, H(\mathbb Q)= I(\mu).
\end{equation}
with the infimum above being taken over all $\mathbb Q$ with marginal $\mu$ must be the increment  of the ergodic Pekar process with generator $L_\psi=\frac 12 \Delta+\frac{\nabla\psi}{\psi}\cdot \nabla$, where $\mu(\d x)=\psi^2(x)\,\d x$. 
The following result, whose proof of is based on Theorem \ref{thm3} and Corollary \ref{cor-thm3}, will finish the proof of Theorem \ref{thm2}. 

\begin{theorem}\label{thm4}
For any $\eps\leq 1$ let $\mathbb Q_\eps\in \Mcal_{\mathrm{si}}(\mathcal X)$ be a maximizer of the variational formula $g(\eps)$, i.e., $\mathbb Q_\eps\in \mathfrak m_\eps$. Then 
$$
\sup_{\eps\leq 1} H(\mathbb Q_\eps |\P)\leq C
$$
 for some constant $C<\infty$. Moreover, if $\mathbb Q_n=\mathbb Q_{\eps_n}$ is any subsequence which converges weakly in $\Mcal_{\mathrm{si}}(\mathcal X)$
to some $\mathbb Q$, then $\mathbb Q$ is the distribution of the increments of a stationary process with marginal $\mu\in \Mcal_1(\R^3)$ such that 
\begin{equation}\label{eq0-thm4}
\lim_{n\to\infty} \eps_n\,\, \E^{\mathbb Q_n} \bigg[\int_0^\infty \frac{\e^{-\eps_n t}\,\, \d t}{|\omega(t)-\omega(0)|}\bigg]= \int\int_{\R^3\times\R^3} \frac{\mu(\d x) \,\mu(\d y)}{|x-y|} >0. 
\end{equation}
\end{theorem}

\begin{proof}
We can repeat the argument in the proof of Lemma \ref{lemma-3-pf-thm1} to show that if $\mathbb Q_\eps$ is a maximizer to $g(\eps)$, then 
$$
\sup_{\eps<1}H(\mathbb Q_\eps |\P) \leq C <\infty,
$$
implying that $(\mathbb Q_\eps)_\eps$ is tight. Let $\mathbb Q_n=\mathbb Q_{\eps_n}$ be a subsequence that converges to some process with stationary increments. 
Again, for each $n$, we can make $\mathbb Q_n[\omega(0)=0]=1$ and view $\mathbb Q_n$ as a sequence of measures in $\mathcal X_0\otimes \R^3$. We consider the average
$$
\mathbb Q_{n,\eps_n}= \eps_n \int_0^\infty \,\d t\,\, \e^{-\eps_n t}\,\, \tau_t \,\mathbb Q_n.
$$
Since $\mathbb Q_{n,\eps_n}$ has the same distribution of increments as $\mathbb Q_n$, and the marginals of $\mathbb Q_n$ in $\mathcal X_0$ are tight, $\mathbb Q_{n,\eps_n}$ has a subsequence,
still denoted by $\mathbb Q_n$, which converges  wea-guely in the space of measures on $\mathcal X_0\otimes\R^3$ to a stationary process $\mathbb Q$, with marginal $\mu\in \Mcal_1(\R^3)$. By lower-semicontinuity, we have 
\begin{equation}\label{eq1-thm4}
H(\mathbb Q| \P) \leq \liminf_{n\to\infty} H(\mathbb Q_n | \P).
\end{equation}
Next, the criterion that characterizes wea-gue convergence, combined with Lemma \ref{lemma-Coulomb} then also imply that, for any $a\in \R^d$, 
$$
\lim_{n\to\infty} \eps_n\,\,\E^{\mathbb Q_n} \bigg[\int_0^\infty \frac{\e^{-\eps_n t} \,\, \d t}{|\omega(t)-\omega(0)|} \bigg] = \int_{\R^3} V(x-a) \,\mu(\d x).
$$
Moreover, if the stationary limit $\mathbb Q$ is ergodic, we also have
$$
\lim_{n\to\infty} \eps_n\,\,\bigg[\int_0^\infty \frac{\e^{-\eps_n t} \,\, \d t}{|\omega(t)-\omega(0)|} \bigg] = \int_{\R^3} V(x-a) \,\mu(\d x),
$$
in probability, and consequently, 
\begin{equation}\label{eq2-thm4}
\lim_{n\to\infty} \eps_n\,\,\E^{\mathbb Q_n}\bigg[\int_0^\infty \frac{\e^{-\eps_n t} \,\, \d t}{|\omega(t)-\omega(0)|} \bigg] = \int\int_{\R^3\times\R^3}  \frac{\mu(\d x)\mu(\d y)}{|x-y|}.
\end{equation}
If $\mathbb Q\in \Mcal_{\mathrm{s}}(\mathcal X_0\otimes \R^d)$ is not ergodic, it is a mixture 
$$
\mathbb Q=\int R\,\, \Gamma(\d R)
$$
of ergodic measures $R\in \Mcal_{\mathrm{s,e}}(\mathcal X_0\otimes \R^3)$, with $\Gamma$ being a suitable probability measure on the space of ergodic measures. If $\mu_R \in \Mcal_1(\R^3)$ denotes the marginal of $R$,
then by \eqref{eq2-thm4},
\begin{equation}\label{eq3-thm4}
\lim_{n\to\infty} \eps_n\,\,\E^{\mathbb Q_n}\bigg[\int_0^\infty \frac{\e^{-\eps_n t} \,\, \d t}{|\omega(t)-\omega(0)|} \bigg] = \int\Gamma(\d R)\,\, \int\int_{\R^3\times\R^3}  \frac{\mu(\d x)\mu(\d y)}{|x-y|}
\end{equation}
On the other hand, by linearity of $H(\cdot|\P)$ and \eqref{eq1-thm4},
\begin{equation}\label{eq4-thm4}
H(\mathbb Q|\P) = H\bigg(\int R \,\, \Gamma(\d R) \big| \P\bigg) = \int H(R|\P) \,\,\Gamma(\d R) \leq \liminf_{n\to\infty} H(\mathbb Q_n | \P).
\end{equation}
Then combining the last two displays we have
$$
\begin{aligned}
\int \Gamma(\d R)\,\, \bigg[\int\int_{\R^3\times\R^3} V(x-y) \,\mu_R(\d x)\mu_R(\d y) - H(R|\P)\bigg] &\geq \lim_{\eps\to 0} \sup_{\mathbb Q\in \Mcal_{\mathrm{si}}} \bigg[\E^{\mathbb Q}\bigg(\int_0^\infty \frac{\eps\e^{-\eps t} \,\, \d t}{|\omega(t)-\omega(0)|}\bigg)- H(\mathbb Q|\P)\bigg] \\
&=\int\int_{\R^3\times\R^3}  \frac{\mu(\d x)\mu(\d y)}{|x-y|} - I(\mu)
 \end{aligned}
$$
by \eqref{eq-DV-free-energy}. Then \eqref{eq-contr-pr} implies \eqref{eq0-thm4}, which combined with Corollary \ref{cor-thm3} implies that $\mathbb Q$ is the distribution of the increments of a stationary 
process with marginal $\mu$, and $\mu(\d x)= \psi^2(x) \d x$ with $\psi\in\mathfrak m$. Then $\mathbb Q=\widehat{\mathbb Q}^{\ssup\psi}$ and Theorem \ref{thm4} is proved.
\end{proof}



\noindent{\bf{Acknowledgement.}} Most of this work was carried out during the first author's long term stay at the Courant Institute, and its hospitality is gratefully acknowledged.

\end{document}